\documentclass{amsart}

\usepackage{amsmath, enumerate, amsfonts, amssymb, amsthm, graphics, graphicx, verbatim, caption, subcaption, MnSymbol, array, bm}

\usepackage[english]{babel}
\usepackage[margin=0.95in]{geometry}
\usepackage{tikz, float}
\usetikzlibrary{shapes}
\usetikzlibrary{positioning}
\usetikzlibrary{decorations.markings}
\usetikzlibrary{arrows}
\tikzstyle{subgroup}=[scale=1]
\tikzstyle{circ}=[draw,shape=circle,fill=blue,scale=0.5]

\newtheorem{theorem}{Theorem}
\newtheorem{proposition}{Proposition}
\newtheorem{corollary}{Corollary}
\newtheorem{lemma}{Lemma}

\def\CD#1{\mathcal{CD}#1}

\newcommand{\mydistance}{.6cm}

\title{\bf Finite Groups with a Trivial Chermak-Delgado Subgroup}

\author{Ryan McCulloch}
\address{Department of Mathematics, University of Bridgeport, Bridgeport, CT 06604}
\email{rmccullo@bridgeport.edu}

\date{\today}

\begin{document}

\begin{abstract}
The Chermak-Delgado lattice of a finite group is a modular, self-dual sublattice of the lattice of subgroups of $G$.  The least element of the Chermak-Delgado lattice of $G$ is known as the Chermak-Delgado subgroup of $G$.  This paper concerns groups with a trivial Chermak-Delgado subgroup.  We prove that if the Chermak-Delgado lattice of such a group is lattice isomorphic to a Cartesian product of lattices, then the group splits as a direct product, with the Chermak-Delgado lattice of each direct factor being lattice isomorphic to one of the lattices in the Cartesian product.  We establish many properties of such groups and properties of subgroups in the Chermak-Delgado lattice.  We define a CD-minimal group to be an indecomposable group with a trivial Chermak-Delgado subgroup.  We establish lattice theoretic properties of Chermak-Delgado lattices of CD-minimal groups.  We prove an extension theorem for CD-minimal groups, and use the theorem to produce twelve examples of CD-minimal groups, each having different CD lattices.  Curiously, quasi-antichain $p$-group lattices play a major role in the author's constructions.
\end{abstract}



\maketitle

\section{Indecomposability and CD-minimal Groups}

Let $G$ be a finite group and $H \leq G$.  Then $m_G(H) = |H||C_G(H)|$ is the Chermak-Delgado measure of $H$ in $G$.  Let $m^*(G) = \max \{ m_G(H) \,\, | \,\, H \leq G \}$ and then define $\CD(G) = \{ H \leq G \,\, | \,\, m_G(H) = m^*(G) \}$.  The subgroup collection $\CD(G)$ forms a sublattice of the lattice of subgroups of $G$.  Furthermore, if $H,K \in \CD(G)$ and $\sigma \in Aut(G)$, then $H^{\sigma} \in \CD(G)$, $\langle H,K \rangle = HK$, $C_G(H) \in \CD(G)$, $C_G(H \cap K) = C_G(H)C_G(K)$, and $C_G(C_G(H)) = H$.  The least element, $T$, of $\CD(G)$ is known as the \textbf{Chermak-Delgado subgroup of} \bm{$G$}, and $T$ is a characteristic, abelian subgroup of $G$ which contains the center of $G$.
 
This modular, self-dual sublattice of the lattice of subgroups of $G$ was first introduced in \cite{Che89}.  Proofs of the properties of $\CD(G)$ stated in the previous paragraph are found in Section 1.G of \cite{Isa08}.  

We lay down some scaffolding that is useful in proving results about the Chermak-Delgado lattice.  Proposition 1 below is a variant on ideas of Ore, see \cite{Ore39}.  Given a nonempty subset $Y \subseteq G$, and an element $g \in G$, we denote $Yg = \{ yg \, | \, y \in Y \}$ and $Y^g = \{ y^g \, | \, y \in Y \}$.  Given $H \leq G$, and $X \subseteq G$ a subset of $G$, $H^X = \langle H^x \, | \, x \in X \rangle$.  When $X = G$ this is the normal closure of $H$ in $G$.  Note that $H^{\emptyset} = 1$, since the subgroup generated by the empty set is the identity subgroup.  

\begin{proposition}\label{prop1.4}
Suppose $H,K \leq G$, and $HH^x = H^xH$ for all $x \in K$.  If $K \leq H^K$, then $K \leq H$.
\end{proposition}

\begin{proof}
First note that the condition that  $HH^x = H^xH$ for all $x \in K$ is equivalent to the condition that  $H^xH^y = H^yH^x$ for all $x,y \in K$.  To see this, let $x,y \in K$.  Then ${(H^xH^y)}^{y^{-1}} = H^{xy^{-1}}H = HH^{xy^{-1}} = {(H^yH^x)}^{y^{-1}}$, and so $H^xH^y = H^yH^x$.  So, for any subset $S$ of $K$, $H^S$ is equal to the product of the subgroups appearing in the join, and the order in which the subgroups appear in the product does not matter. 

Suppose by way of contradiction that $K \leq H^K$ and $K \nleq H$.  Let $S$ be a minimal subset of $K$ so that $H^S = H^K$; and choose $S$ so that $S$ contains the identity.  By minimal subset of $K$ so that $H^S = H^K$, we mean that $H^S = H^K$ and for any proper subset, $T$, of $S$, $H^T \neq H^K$.  If $M$ is a minimal subset of $K$ so that $H^M = H^K$, then $M$ is nonempty.  The is true because otherwise $1 = H^{\emptyset}=H^K$ implies that $H=1$, and $K \leq H^{\emptyset} = 1$ implies that $K=1$, and so $K \leq H$, contrary to assumption.  Any such $M$ contains more than one element, as otherwise $K \leq H^k$ implies that $K = K^{k^{-1}} \leq H$, contrary to assumption.  Now we argue why we can choose $S$ to contain the identity.  Suppose $R = \{ y_1 , y_2, ... , y_m \}$ is a minimal subset of $K$ so that $H^R = H^K$.  Then $R{y_1}^{-1}$ contains the identity and is also a minimal subset of $K$ so that $H^{R{y_1}^{-1}} = H^K$.  This is true because $H^{R{y_1}^{-1}} = {(H^R)}^{{y_1}^{-1}} = {(H^K)}^{{y_1}^{-1}} = H^K$; and any proper subset, $P$, of $R{y_1}^{-1}$ has the form $P = T{y_1}^{-1}$ where $T$ is some proper subset of $R$, and so $H^{{P}} = H^{T{y_1}^{-1}} = {(H^T)}^{{y_1}^{-1}} \neq {(H^K)}^{{y_1}^{-1}} = H^K$. 

So $|S|$ is at least 2 and let $S = \{ 1, x_1, ... , x_n \}$.  Let $U= HH^{x_2}\cdots H^{x_n}$ so that $K \leq H^S = UH^{x_1}$.  So $x_1 = uh^{x_1}$ for some $u \in U$ and $h \in H$.  So $x_1 = u{x_1}^{-1}h{x_1}$, and so $1 =  u{x_1}^{-1}h$, and so ${x_1}^{-1} = {u}^{-1}{h}^{-1} = {(hu)}^{-1}$, and so $x_1 = hu \in U$.  Thus $H^{x_1} \leq U = HH^{x_2} \cdots H^{x_n}$, which implies that $H^{S - \{x_1\}} = H^K$, contradicting the minimality of $S$.
\end{proof}

\begin{corollary}\label{cor1.1}
Suppose $H,K \leq G$, and $HH^x = H^xH$ for all $x \in K$.  If $H < K$, then $H^K < K$.
\end{corollary}

\begin{proof}
If $H < K$, then $H^K \leq K$.  Now $H < K$ implies that $K \nleq H$ so by the contrapositive of Proposition~\ref{prop1.4}, $K \nleq H^K$, and so in particular $H^K \neq K$.  So $H^K < K$.
\end{proof}

If $H \in \CD(G)$, then any $x \in G$ induces an automorphism of $G$, and so $H^x \in \CD(G)$, and so $\langle H , H^x \rangle = HH^x$, i.e. $HH^x = H^xH$.  Thus Proposition~\ref{prop1.4} and Corollary~\ref{cor1.1} apply for any subgroup $H \in \CD(G)$ and any subgroup $K$ of $G$.

Given $H \leq G$ and $X \subseteq G$, we use $core_X(H)$ to denote intersection of all of the conjugates of $H$ in $X$.  We note that if $H \in \CD(G)$ and $\emptyset \neq X \subseteq G$, then $H^X$ and $core_X(H)$ are both in $\CD(G)$.  

\begin{proposition}\label{prop1.5}
\begin{enumerate}
\item If $H \in \CD(G)$ with $H < G$, then $H^G < G$.
\item If $K \in \CD(G)$ with $Z(G) < K$, then $Z(G) < core_G(K)$.
\end{enumerate}
\end{proposition}

\begin{proof}
Part (1) follows directly from Corollary~\ref{cor1.1}.

If $Z(G) < K$, then $C_G(K) < G$, so by Corollary~\ref{cor1.1}, ${C_G(K)}^G < G$.  And so $Z(G) < C_G({C_G(K)}^G) = core_G(C_G(C_G(K)) = core_G(K)$.  
\end{proof}

Given $H,K \in \CD(G)$, we use the notation $H \prec K$ to mean that $H < K$ and there is no $R \in \CD(G)$ so that $H < R < K$.  If $M$ is the greatest element in $\CD(G)$ and $T$ is the least element in $\CD(G)$ (i.e. $T$ is the Chermak-Delgado subgroup of $G$), we say that $A \in \CD(G)$ is an \textbf{atom} if $T \prec A$, and we say that $B \in \CD(G)$ is a \textbf{coatom} if $B \prec M$.  Proposition 3 below appears in \cite{Bre12}.  The proof in \cite{Bre12} is different than our approach.

\begin{proposition}
Let $H,K \in \CD(G)$ with $H \prec K$.  Then $H \unlhd K$.  And so maximal chains in $\CD(G)$ form subnormal series of $G$.
\end{proposition}

\begin{proof}
$H < K$ and so by Corollary~\ref{cor1.1}, $H^K < K$, and we know that $H^K \in \mathcal{C}\mathcal{D}(G)$.  Since $H \prec K$, it must be that $H^K = H$, i.e. $H \unlhd K$.  

Now the greatest element in $\mathcal{C}\mathcal{D}(G)$ is normal in $G$ (in fact characteristic in $G$), and so maximal chains in $\CD(G)$ form subnormal series of $G$.
\end{proof}

\begin{corollary}\label{cor1.2}
Suppose $G \in \CD(G)$.  Then all of the atoms and all of the coatoms of $\CD(G)$ are normal in $G$.  
\end{corollary}

\begin{proof}
If $B \in \CD(G)$ is a coatom, then $B \prec G$, and so by Proposition 3, $B \unlhd G$.  If $A \in \CD(G)$ an atom, then $C_G(A) \in \CD(G)$ is a coatom, and so $C_G(A) \unlhd G$, and so $C_G(C_G(A)) = A \unlhd G$. 
\end{proof}

The topic of this paper are those groups whose Chermak-Delgado subgroup is the identity subgroup.  Consider direct products.  Given $H \leq G_1 \times \dots \times G_n = G$, one sees that $C_G(H) = C_G({\pi}_1(H) \times \dots \times {\pi}_n(H))$, where ${\pi}_i$ is the projection map into the $i$-th coordinate.  From here one sees that the Chermak-Delgado lattice of a direct product is the Cartesian product of the Chermak-Delgado lattices.  Proposition 4 appears in \cite{Bre12}:

\begin{proposition} $\CD(G_1 \times \dots \times G_n) = \CD(G_1) \times \dots \times \CD(G_n)$.
\end{proposition}

\begin{corollary} $1 \in \CD(G_1 \times \dots \times G_n)$ if and only if $1 \in \CD(G_i)$ for each $i$.
\end{corollary}

\begin{proposition}
Suppose $1 \in \CD(G)$.  If $H,K \in \CD(G)$ so that $H \cap K = 1$ and $G=HK$, then $G = H \times K$ is a direct product.
\end{proposition}

\begin{proof}
Suppose $H,K \in \CD(G)$ so that $H \cap K = 1$ and $G=HK$.  Let $A = core_K(Z(H))$ and let $B = core_A(K)$.  Note that $A$ and $B$ normalize one another, and $A \cap B = 1$, and thus $A = core_K(Z(H)) \leq C_G(core_A(K)) = C_G(K)^A$.  So by Proposition 1, $A \leq C_G(K)$.  And since $A \leq Z(H)$ and $G=HK$, we have $A \leq Z(G) = 1$.  And so $1 = core_K(Z(H)) = core_G(Z(H))$, and so by Proposition 2 part 2, $Z(H) = 1$.  Now $H \in \CD(G)$ and $m^*(G) = |G|$.  And so $|G| = |H||C_G(H)| = \dfrac{|H||C_G(H)|}{|Z(H)|} = |HC_G(H)|$, and so $G = HC_G(H)$, and so $H \unlhd G$.  By a similar argument, switching $H$ with $K$, one obtains that $K \unlhd G$.  And so $G = H \times K$ is a direct product.
\end{proof}

\begin{corollary}\label{prop1.17}
Suppose $1 \in \CD(G)$.  If $G = AB$ where $A$ and $B$ are abelian subgroups of $G$, then $G = 1$.
\end{corollary}
\begin{proof}
Since $A$ and $B$ are abelian, $A \cap B \leq Z(G) = 1$.  So, $$|A||B| = |G| = m_G(G) \geq m_G(A) = |A||C_G(A)| \geq |A||A| \textrm{,}$$ so $|B| \geq |A|$.  Similarly, $$|A||B| = |G| = m_G(G) \geq m_G(B) = |B||C_G(B)| \geq |B||B| \textrm{,}$$ so $|A| \geq |B|$; and it follows that $|A|=|B|$, and we have equalities everywhere in the previous two strings of inequalities.  So $A,B \in \mathcal{C}\mathcal{D}(G)$ and so by Proposition 5, $G = A \times B$, and so $G$ is abelian, and so $1 = Z(G) = G$.
\end{proof}

\begin{theorem}
Suppose $1 \in \CD(G)$, and suppose $\CD(G) \cong \mathcal{L}_1 \times \mathcal{L}_2$ for lattices $\mathcal{L}_1$ and $\mathcal{L}_2$.  Then $G = H \times K$ for subgroups $H$ and $K$ with $\CD(H) \cong \mathcal{L}_1$ and $\CD(K) \cong \mathcal{L}_2$.
\end{theorem}

\begin{proof}
Suppose $1 \in \CD(G)$ and suppose that $\CD(G)$ is lattice isomorphic to a Cartesian product $\mathcal{L}_1 \times \mathcal{L}_2$ of lattices $\mathcal{L}_1$ and $\mathcal{L}_2$.  Let $M_1$, $B_1$ be the greatest and least elements, respectively, of $\mathcal{L}_1$ and let $M_2$, $B_2$ be the greatest and least elements, respectively, of $\mathcal{L}_2$.  Note that $(B_1,M_2) \vee (M_1,B_2) = (M_1,M_2)$ is the greatest element of $\mathcal{L}_1 \times \mathcal{L}_2$ and $(B_1,M_2) \wedge (M_1,B_2) = (B_1,B_2)$ is the least element of $\mathcal{L}_1 \times \mathcal{L}_2$.  And since $\CD(G)$ is lattice isomorphic to $\mathcal{L}_1 \times \mathcal{L}_2$, there is $H,K \in \CD(G)$ corresponding to $(M_1,B_2),(B_1,M_2)$, respectively, in $\mathcal{L}_1 \times \mathcal{L}_2$ so that $H \cap K = 1$ and $HK=G$.  By Proposition 5, $G = H \times K$ is a direct product.  By Proposition 4, $\CD(G) = \CD(H) \times \CD(K)$.  And so we have a lattice isomorphism between $\CD(H) \times \CD(K)$ and $\mathcal{L}_1 \times \mathcal{L}_2$ where $H$ corresponds to $(M_1,B_2)$ and $K$ corresponds to $(B_1,M_2)$.  Thus $\CD(H) \cong \mathcal{L}_1$ and $\CD(K) \cong \mathcal{L}_2$.
\end{proof}

By induction, this theorem extends to a direct product/Cartesian product of $n$ number of groups/CD lattices.

A group $G$ is said to be \textbf{indecomposable} if $G$ cannot be written as an (internal) direct product $H \times K$ with $H \neq 1$ and $K \neq 1$.  A lattice $\mathcal{L}$ is said to be \textbf{indecomposable} if $\mathcal{L}$ is not lattice isomorphic to a Cartesian product $\mathcal{L}_1 \times \mathcal{L}_2$ of lattices with $\mathcal{L}_1$ and $\mathcal{L}_2$ both nontrivial (a trivial lattice is a lattice consisting of a single point).

\begin{corollary}
Suppose $1 \in \CD(G)$.  Then $G$ is indecomposable if and only if $\CD(G)$ is indecomposable.
\end{corollary}

\begin{proof}
Suppose $1 \in \CD(G)$ and suppose $\CD(G)$ is indecomposable.  If $G = H \times K$, then by Proposition 4, $\CD(G) = \CD(H) \times \CD(K)$.  Since $\CD(G)$ is indecomposable, at least one of $\CD(H)$ or $\CD(K)$ is trivial, say without loss of generality that $\CD(H)$ is trivial.  By Corollary 3, since $1 \in \CD(G)$, $1 \in \CD(H)$, and so $\CD(H) = \{1\}$.  And so $H = 1$, and thus $G$ is indecomposable.

Suppose $1 \in \CD(G)$ and suppose $G$ is indecomposable.  Suppose that $\CD(G)$ is lattice isomorphic to a Cartesian product $\mathcal{L}_1 \times \mathcal{L}_2$ of lattices $\mathcal{L}_1$ and $\mathcal{L}_2$.  By Theorem 1, $G = H \times K$ for subgroups $H$ and $K$ with $\CD(H) \cong \mathcal{L}_1$ and $\CD(K) \cong \mathcal{L}_2$.  And since $G$ is indecomposable, at least one of $H$ or $K$ is trivial, and so at least one of $\mathcal{L}_1$ or $\mathcal{L}_2$ is trivial, and so $\CD(G)$ is indecomposable.
\end{proof}

In \cite{McC17}, the author studied \textbf{CD-simple} groups, which are groups, $G$, having the property that $\CD(G) = \{1, G \}$.  We define a group, $G$, to be \textbf{CD-minimal} if $1 \in \CD(G)$ and $G$ is indecomposable.  So every CD-simple group is CD-minimal, but not vice versa.

\begin{proposition}
Suppose $G$ is CD-minimal but not CD-simple.  If $A \in \CD(G)$ is an atom and if $B \in \CD(G)$ is a coatom, then $A \leq B$.
\end{proposition}

\begin{proof}
By Corollary 2, both $1 \neq A$ and $1 \neq B$ are normal in $G$.  If $A \cap B = 1$, then $G = A \times B$ is a direct product, contrary to the assumption that $G$ is indecomposable.  So $A \cap B \neq 1$, and since $A$ is an atom in $\CD(G)$, $A \leq B$.
\end{proof}

\begin{proposition} Suppose $G$ is CD-minimal but not CD-simple.  If $A \in \CD(G)$ is an atom, then $A$ is abelian, $A \unlhd G$, and $|A|$ contains primes $p \neq q$.
\end{proposition}

\begin{proof}
Let $1 \neq A \in \CD(G)$ an atom.  Then $C_G(A) \in \CD(G)$ is a coatom, and by Proposition 6, $A \leq C_G(A)$, i.e. $A$ is abelian.  By Corollary 2, $A \unlhd G$.  Let $N \leq A$ be a minimal normal subgroup of $G$.  So $|N| = p^k$ for some prime $p$ and some $k$.  Now, as $N$ is minimal normal, $G/C_G(N)$ acts faithfully and irreducibly on $N$, and so $|G/C_G(N)|$ is not a power of $p$.  And so a prime $q\neq p$ divides $|G/C_G(N)|$.  Now $|G/C_G(N)|$ divides $|G/C_G(A)|$, and since $A \in \CD(G)$, $|A| = |G/C_G(A)|$.  Thus $|A|$ contains primes $p \neq q$.
\end{proof}

\begin{lemma} If $H,K \in \CD(G)$ and $K \leq H$, then $|H:K| = |C_G(K) : C_G(H)|$.
\end{lemma}

\begin{proof}
As $H,K \in \CD(G)$, $|H||C_G(H)| = |K||C_G(K)|$, and so $|H:K| = |C_G(K) : C_G(H)|$.
\end{proof}

Corollary 6 follows immediately:

\begin{corollary}  Suppose $G$ is CD-minimal but not CD-simple.  If $B \in \CD(G)$ is a coatom, then $C_G(B) \leq B$, $B \unlhd G$, and $|G:B|$ contains primes $p \neq q$
\end{corollary}

\begin{corollary}
Suppose $1 \in \CD(G)$.  If $1 \neq H \in \CD(G)$, then $H$ is not a $p$-group for any prime $p$.
\end{corollary}

\begin{proof}
If we prove Corollary 7 for any CD-minimal group, then by Corollary 3, the result will follow for any arbitrary group $X$ with $1 \in \CD(X)$.  So suppose $G$ is a CD-minimal group.  If $G$ is CD-simple, then $G$ is not a $p$-group for any prime $p$ since $Z(G) = 1$.  If $G$ is CD-minimal and not CD-simple, and $1 \neq H \in \CD(G)$, then $H$ contains an atom $A \in \CD(G)$, and so by Proposition 7, $H$ is not a $p$-group for any prime $p$.
\end{proof}

\begin{corollary} Suppose $1 \in \CD(G)$.  If $G \neq H \in \CD(G)$, then $|G:H|$ contains at least two different primes, and it follows that $H$ is not a maximal subgroup of $G$.
\end{corollary}

\begin{proof} 
If we prove Corollary 8 for any CD-minimal group, then by Corollary 3, the result will follow for any arbitrary group $X$ with $1 \in \CD(X)$.  So suppose $G$ is a CD-minimal group.  If $G$ is CD-simple, then $1 = Z(G) = H \neq G$ implies that $|G| = |G:1|$ contains at least two different primes.  So suppose $G$ is CD-minimal and not CD-simple.  Suppose $G \neq H \in \CD(G)$.  Then $H \leq B$ where $B \in \CD(G)$ is a coatom.  $|G : B|$ divides $|G : H|$, and by Corollary 6, $|G : B|$ contains primes $p \neq q$, and so  $|G : H|$ contains primes $p \neq q$.  Suppose by way of contradiction that $H$ is a maximal subgroup of $G$ and let $P$ be a Sylow $p$-subgroup of $G$.  Then $H < HP$.  Since $H$ is maximal in $G$, $HP = G$.  But $|G : HP|$ contains a prime $q$, a contradiction.
\end{proof} 

Let us denote the class of all finite groups with a trivial Chermak-Delgado subgroup by $\mathcal{T}$.  We summarize what is known about the class $\mathcal{T}$.  The class $\mathcal{T}$ is closed under direct products by Corollary 4, and if $G \in \mathcal{T}$ is a direct product, then each direct factor of $G$ is in $\mathcal{T}$, also by Corollary 4.  Of course, $G \in \mathcal{T}$ implies that $Z(G) = 1$, and so the intersection of $\mathcal{T}$ with the class of nilpotent groups is trivial.  $\mathcal{T}$ contains the class of all non-abelian simple groups.  Group theoretic properties of CD-simple groups are studied in \cite{McC17}, and examples of solvable CD-simple groups can be found there (the symmetric group on 4 elements is the smallest nontrivial example).  $G \in \mathcal{T}$ cannot factor non-trivially as a product of abelian subgroups by Corollary 4.  

Given $G \in \mathcal{T}$, we summarize what is known about the properties of the groups in $\CD(G)$.  Given $G \in \mathcal{T}$, $\CD(G)$ does not contain any non-trivial $p$-groups by Corollary 7, and $\CD(G)$ does not contain any subgroups whose index in $G$ is a (nontrivial) power of a prime by Corollary 8, and this implies that $\CD(G)$ does not contain any maximal subgroups of $G$.  It was shown in \cite{McC17} (see Proposition 6 there) that given $G \in \mathcal{T}$, $\CD(G)$ does not contain any non-trivial, cyclic subgroups that are normal in $G$.  Note that we can apply Proposition 2 part 2 to show that given $G \in \mathcal{T}$, $\CD(G)$ does not contain any non-trivial, cyclic groups.  Given $G \in \mathcal{T}$, we know that all of the atoms and all of the coatoms in $\CD(G)$ are normal subgroups of $G$ by Corollary 2; and if $G$ is not CD-simple, we know that all of the atoms in $\CD(G)$ are abelian groups by Proposition 7. 

Let $\mathcal{T}^*$ denote the class of all lattices $\CD(G)$ with $G \in \mathcal{T}$.  We summarize what is known about the lattice theoretic properties of $\mathcal{T}^*$. We know that $\CD(G)$ is modular and self-dual for any finite group $G$, and so these properties hold in particular for all lattices $\mathcal{L} \in \mathcal{T}^*$.  The class $\mathcal{T}^*$ is closed under Cartesian products by Corollary 4, and if $\mathcal{L} \in \mathcal{T}^*$ is lattice isomorphic to a Cartesian product of lattices, then each lattice appearing in that Cartesian product is in $\mathcal{T}^*$; this is by Theorem 1.  Given indecomposable $\mathcal{L} \in \mathcal{T}^*$ with $|\mathcal{L}| > 2$, and given $X \in \mathcal{L}$ an atom, and $Y \in \mathcal{L}$ a coatom, we have that $X \leq Y$ by Proposition 6.   A lattice $\mathcal{L}$ is said to be a \textbf{quasi-antichain of width} \bm{$n$} if $\mathcal{L}$ contains exactly $n$ atoms, and every atom in $\mathcal{L}$ is a coatom in $\mathcal{L}$.  We denote a quasi-antichain lattice of width $n$ by $\mathcal{M}_n$.  And so given indecomposable $\mathcal{L} \in \mathcal{T}^*$, $\mathcal{L}$ is not isomorphic to $\mathcal{M}_n$ for any $n > 1$.  $1 = G \in \mathcal{T}$ has that $\CD(G) = \{1\} \cong  \mathcal{M}_0$, and in the next section we construct a group $G \in \mathcal{T}$ so that $\CD(G) \cong \mathcal{M}_1$, but beyond that, $\CD(G)$ is never quasi-antichain for indecomposable $G \in \mathcal{T}$.  This is surprising considering the prevalence of $\mathcal{M}_n$ lattices in the theory of Chermak-Delgado $p$-group lattices, see \cite{Bree14} and \cite{An17}.  We will see in the next section that quasi-antichain $p$-group lattices do play a major role in the CD lattices of some CD-minimal group examples.

Suppose $\mathcal{L}$ is a lattice with greatest element $M$ and least element $B$.  Given $X \in \mathcal{L}$, we say that $Y \in \mathcal{L}$ is a \textbf{complement of} \bm{$X$} if $X \vee Y = M$ and $X \wedge Y = B$.  Given indecomposable $\mathcal{L} \in \mathcal{T}^*$, it follows from Proposition 5 that the only elements in $\mathcal{L}$ that have complements in $\mathcal{L}$ are the greatest and least elements in $\mathcal{L}$.  This further restricts the structure of indecomposable $\mathcal{L} \in \mathcal{T}^*$. 

\section{Examples of CD-Minimal Groups which are not CD-Simple}

Do there exist CD-minimal groups which are not CD-simple?  The answer is yes!  

$\mathbb{Z}_n$ denotes the group of integers modulo $n$, $C_n$ denotes the cyclic group of order $n$, $S_n$ denotes the symmetric group on $n$ elements, $Q_8$ denotes the quaternion group of order $8$, and $QD_{16}$ denotes the quasidihedral group of order 16.  Note that $QD_{16} = \langle r, s \,\, | \,\, r^8 = s^2 = 1, srs = r^3 \rangle$.  A group action of a group $T$ on a group $N$ is a group homomorphism $T \rightarrow Aut(N)$. The group action is said to be faithful if the kernel of the homomorphism is trivial.  The group action is said to be irreducible if there are no proper, nontrivial $T$-invariant subgroups of $N$.  Given a group action of a group $T$ on a group $N$, we let $[N]T$ denote the semidirect product.

\begin{lemma}
Suppose $1 \neq P$ is a $p$-group and $T$ is a group so that $T$ acts on $P$ and so that so that the restricted action of $T$ on $Z(P)$ is faithful and irreducible, and let $G = [P]T$.  If $1 \neq A \unlhd G$ and $A$ is abelian, then $Z(P) \leq A \leq C_G(A) \leq P$.
\end{lemma}

\begin{proof}
Note that if $A \cap Z(P) = 1$ then since $A \unlhd G$, and $Z(P) \unlhd G$, we have $A \leq C_G(Z(P)) = P$, since $T$ acts faithfully on $Z(P)$.  But since $1 \neq A \unlhd P$ and $1 \neq P$ is a $p$-group, we have $A \cap Z(P) > 1$, a contradiction.  So $A \cap Z(P) > 1$.  Now, since $T$ acts irreducibly on $Z(P)$, $Z(P) \leq A$.  And so $A \leq C_G(A) \leq C_G(Z(P)) = P$.
\end{proof}

For $i \in \{1,2\}$, let $P_i$ be a $2$-group so that $P_i \in \CD(P_i)$, so that $Z(P_i) \cong \mathbb{Z}_2 \times \mathbb{Z}_2$, and so that there exists a group $T_i = [H_i]K_i$ with $|H_i| = 3$ and $|K_i| = 2$ so that $T_i$ acts on $P_i$ and so that the restricted action of $T_i$ on $Z(P_i)$ is faithful and irreducible (and so $T_i \cong Aut(\mathbb{Z}_2 \times \mathbb{Z}_2) \cong S_3$).

Let $P_3$ be a $3$-group so that $P_3 \in \CD(P_3)$, so that $Z(P_3) \cong \mathbb{Z}_3 \times \mathbb{Z}_3$, and so that there exists a group $T_3 = [H_3]K_3$ with $|H_3| = 8$ and $|K_3| = 2$ so that $T_3$ acts on $P_3$ and so that the restricted action of $T_3$ on $Z(P_3)$ is faithful and irreducible (note that a Sylow $2$-subgroup of $Aut(\mathbb{Z}_3 \times \mathbb{Z}_3)$ is isomorphic to $QD_{16} = [\langle r \rangle]\langle s \rangle$.)

Let $G_1 = [P_1]T_1$, $G_2 = [P_2]T_2$, $G_3 = [P_3]T_3$, and let $G = G_1 \times G_2 \times G_3$.

Let ${\pi}_i$ denote the projection homomorphism of $G$ onto the $i$th coordinate.  Let $S \leq G$ so that for each $i \in \{1,2,3\}$, ${\pi}_i(S) =G_i$ and $S \cap G_i = [P_i]H_i$. So $S \cap (K_1 \times K_2 \times K_3)$ is the diagonal subgroup of $K_1 \times K_2 \times K_3$.

\begin{lemma}
Let $S$ be a group constructed as above.  If $H \in \CD(S)$, and if for each $i \in \{1,2,3\}$, $Z(P_i) \leq {\pi}_i(H) \leq P_i$, then $H, C_S(H) \in (\CD(P_1) \times \CD(P_2) \times \CD(P_3))$.
\end{lemma}

\begin{proof}
Note that $C_S(H) = C_S({\pi}_1(H) \times {\pi}_2(H) \times {\pi}_3(H)) \leq C_S(Z(P_1) \times Z(P_2) \times Z(P_3)) = P_1 \times P_2 \times P_3$ since each $T_i$ acts faithfully on $Z(P_i)$.  And so $C_S(H) = C_{P_1}({\pi}_1(H)) \times C_{P_2}({\pi}_2(H)) \times C_{P_3}({\pi}_3(H))$.  And since $H \in \CD(S)$, $H = {\pi}_1(H) \times {\pi}_2(H) \times {\pi}_3(H)$.  And so $H, C_S(H) \in (\CD(P_1) \times \CD(P_2) \times \CD(P_3))$.
\end{proof}

Theorem 2 below is an example of a CD lattice extension theorem.  There have been papers written about CD lattice extension theorems for $p$-groups, see \cite{Bre14}, \cite{An15}.  This is the first non $p$-group CD lattice extension theorem that the author is aware of.

\begin{theorem}
Let $S$ be a group constructed as above, and suppose $|P_1| = 2^a$, $|P_2| = 2^b$, and $|P_3| = 3^c$ for some $a,b,c$.  Then $S$ is a CD-minimal group of order $2^{a+b+4} \cdot 3^{c+2}$ and $\CD(S) = (\CD(P_1) \times \CD(P_2) \times \CD(P_3)) \cup \{1,S\}$.
\end{theorem}

\begin{proof}
Let $S$ be a group constructed as above, and suppose $|P_1| = 2^a$, $|P_2| = 2^b$, and $|P_3| = 3^c$ for some $a,b,c$.  And so $|Z(P_1)| = 4$, $|Z(P_2)| = 4$, and $|Z(P_3)| = 9$.  And since $P_1 \in \CD(P_1)$, $P_2 \in \CD(P_2)$, and $P_3 \in \CD(P_3)$,  we have that $2^{a+b+4} \cdot 3^{c+2} = m^*(P_1) \cdot m^*(P_2) \cdot m^*(P_3) = |S| = |P_1||H_1||P_2||H_2||P_3||H_3|2$.

We now establish that the Chermak-Delgado subgroup of $S$ is $1$, and along the way we determine all of the abelian, normal subgroups of $S$ that are in $\CD(S)$.  Note that $m_S(1) = |S|$.  We know that the Chermak-Delgado subgroup of $S$ is a characteristic, abelian subgroup of $S$, and so let us suppose that $1 \neq A \unlhd S$ with $A$ abelian and $A \in \CD(S)$.  Since $A \unlhd S$, for each $i \in \{1,2,3\}$, ${\pi}_i(A) \unlhd {\pi}_i(S) = G_i$.  By Lemma 2, if ${\pi}_i(A) \neq 1$, then $Z(P_i) \leq {\pi}_i(A) \leq C_{G_i}({\pi}_i(A)) \leq P_i$.  And so $C_{G_i}({\pi}_i(A)) = {\pi}_i(C_S(A))$ and $|{\pi}_i(A)||{\pi}_i(C_S(A))| \leq m^*(P_i)$.  We now show that for each $i \in \{1,2,3\}$, ${\pi}_i(A) \neq 1$.  Suppose not.  Then there is $j_0 \in \{1,2,3\}$ so that ${\pi}_{j_0}(A) = 1$.  And since $1 \neq A$, there is $i_0 \in \{1,2,3\}$ so that ${\pi}_{i_0}(A) \neq 1$.  And so if ${\pi}_{j}(A) = 1$, then ${\pi}_j(C_S(A)) = [P_j]H_j$.  This is true because $T_{i_0}$ acts faithfully on $Z(P_{i_0}) \leq {\pi}_{i_0}(A)$, and so the diagonal subgroup of $K_1 \times K_2 \times K_3$ will not centralize $A$.  And so if ${\pi}_j(A) = 1$, then $|{\pi}_j(C_S(A))| = |P_j||H_j| < |P_j||Z(P_j)| = m^*(P_j)$.  And so $m_S(A) = |A||C_S(A)| \leq |{\pi}_1(A)||{\pi}_2(A)||{\pi}_3(A)||{\pi}_1(C_S(A))||{\pi}_2(C_S(A))||{\pi}_3(C_S(A))| < m^*(P_1) \cdot m^*(P_2) \cdot m^*(P_3) = |S|$, which contradicts $A \in \CD(S)$.

So for each $i \in \{1,2,3\}$, ${\pi}_i(A) \neq 1$.  And so by Lemma 2, for each $i \in \{1,2,3\}$, $Z(P_i) \leq {\pi}_i(A) \leq P_i$.  And so by Lemma 3, $A, C_S(A) \in (\CD(P_1) \times \CD(P_2) \times \CD(P_3))$.  And so $m_S(A) = m^*(P_1) \cdot m^*(P_2) \cdot m^*(P_3) = |S|$.  Hence we have shown that the Chermak-Delgado subgroup of $S$ is the identity, and hence $m^*(S) = |S|$, and so $(\CD(P_1) \times \CD(P_2) \times \CD(P_3)) \cup \{1,S\} \subseteq \CD(S)$.  Furthermore, we know that if $1 \neq A \unlhd S$ with $A$ abelian and $A \in \CD(S)$, then $A \in (\CD(P_1) \times \CD(P_2) \times \CD(P_3))$.  

Note that $Z(P_1) \times Z(P_2) \times Z(P_3)$ is an atom in $\CD(S)$.  This is true because otherwise, $Z(P_1) \times Z(P_2) \times Z(P_3)$ would properly contain an atom $1 \neq A$ of $\CD(S)$.  By Corollary 2, $A \unlhd S$.  But then $A \in (\CD(P_1) \times \CD(P_2) \times \CD(P_3))$, a contradiction.  We now show that $Z(P_1) \times Z(P_2) \times Z(P_3)$ is the unique atom in $\CD(S)$.  Suppose not.  So $1 \neq N$ is another atom in $\CD(S)$.  So by Corollary 2, $N \unlhd S$.  And since $N \cap (Z(P_1) \times Z(P_2) \times Z(P_3)) = 1$, we have that $N \leq C_S(Z(P_1) \times Z(P_2) \times Z(P_3)) = P_1 \times P_2 \times P_3$.  And since $P_1 \times P_2 \times P_3$ is nilpotent, we have that $N \cap (Z(P_1) \times Z(P_2) \times Z(P_3)) > 1$, a contradiction.  So $Z(P_1) \times Z(P_2) \times Z(P_3)$ is the unique atom in $\CD(S)$.

If $H \in \CD(S)$ with $H \neq 1$ and $H \neq S$, then $Z(P_1) \times Z(P_2) \times Z(P_3) \leq H,C_S(H)$, and so $C_S(H),H \leq P_1 \times P_2 \times P_3$, and so by Lemma 3, $H \in (\CD(P_1) \times \CD(P_2) \times \CD(P_3))$.  Thus $\CD(S) = (\CD(P_1) \times \CD(P_2) \times \CD(P_3)) \cup \{1,S\}$.  

Finally, since $\CD(S)$ contains a unique atom, we have that $\CD(S)$ is an indecomposable lattice.  And so by Corollary 5, $S$ is indecomposable.  Thus, $S$ is a CD-minimal group with the prescribed properties.
\end{proof}

We desire examples of groups $P_1, P_2, P_3$ and acting groups $T_1, T_2, T_3$ in Theorem 2.  One can take $P_1 = \mathbb{Z}_2 \times \mathbb{Z}_2$, $P_2 = \mathbb{Z}_2 \times \mathbb{Z}_2$, $P_3 = \mathbb{Z}_3 \times \mathbb{Z}_3$, and $T_1 = Aut(\mathbb{Z}_2 \times \mathbb{Z}_2)$, $T_2 = Aut(\mathbb{Z}_2 \times \mathbb{Z}_2)$, and $T_3$ to be a Sylow $2$-subgroup of $Aut(\mathbb{Z}_3 \times \mathbb{Z}_3)$.  Let $S_0$ be the CD-minimal group constructed using each of these examples.  The author thanks Peter Hauck for providing this construction. The group $S_0$ has order $2^8 \cdot 3^4 = 20736$.  Figure 1 shows the Chermak-Delgado lattice of $S_0$.

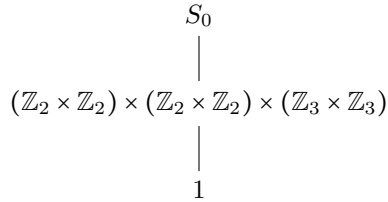
\begin{figure}[h]
\begin{tikzpicture}[node distance=2cm]
\title{Chermak-Delgado lattice}
\node(S)                           {$S_0$};
\node(A)       [below=\mydistance of S] {$( \mathbb{Z}_2 \times \mathbb{Z}_2) \times ( \mathbb{Z}_2 \times \mathbb{Z}_2) \times ( \mathbb{Z}_3 \times \mathbb{Z}_3)$};
\node(1)      [below=\mydistance of A]  {$1$};
\draw(S)      -- (A);
\draw(A)      -- (1);
\end{tikzpicture}
\caption{Chermak-Delgado lattice of $S_0$}
\end{figure}

And so $\CD(S_0) \cong \mathcal{M}_1$.  Note that the unique atom/coatom, $A$, in $\CD(S_0)$ is abelian, $A \unlhd S_0$, and $|A|$ contains at least two primes.  This (thankfully) agrees with the theory established in Section 1.

Proposition 8 below appears in \cite{Bree14}.  The proof makes use of Exercise 39 in \cite{Hup67}.
\begin{proposition}
Let $p$ be a prime and $n$ a positive integer.  Let $P$ be the group of all $3 \times 3$ lower triangular matrices over $GF(p^n)$ with $1$'s along the diagonal.  The Chermak-Delgado lattice of $P$ is a quasi-antichain of width $p^n+1$ and all subgroups in the middle antichain are abelian.
\end{proposition}

\begin{figure}[h]
\begin{tikzpicture}[node distance=2cm]
\title{Chermak-Delgado lattice}
\node(S)                           {$P$};
\node(A3)       [below=0.65cm of S] {$A_3$};
\node(A2)       [left=0.2cm of A3] {$A_2$};
\node(A1)       [left=0.2cm of A2] {$A_1$};
\node(A4)       [right=0.2cm of A3] {$A_4$};
\node(A5)       [right=0.2cm of A4] {$A_5$};
\node(1)      [below=\mydistance of A]  {$Z(P)$};
\draw(S)      -- (A3);
\draw(S)      -- (A2);
\draw(S)      -- (A1);
\draw(S)      -- (A4);
\draw(S)      -- (A5);
\draw(A3)      -- (1);
\draw(A2)      -- (1);
\draw(A1)      -- (1);
\draw(A4)      -- (1);
\draw(A5)      -- (1);
\end{tikzpicture}
\caption{$\CD(P)$ where $P$ is the group of all $3 \times 3$ lower triangular matrices over $GF(4)$ with $1$'s along the diagonal. }
\end{figure}
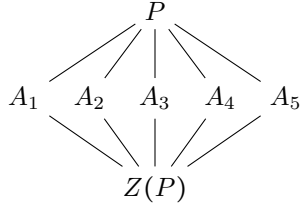

\begin{proposition}
Let $p$ be a prime and $n$ a positive integer.  Let $P$ be the group of all $3 \times 3$ lower triangular matrices over $GF(p^n)$ with $1$'s along the diagonal. Then there exists $T \leq Aut(P)$ so that $T = [H]K$ with $|H| = p^n - 1$ and $|K| = n$, and so that the restricted action of $T$ on $Z(P) \cong \underset{n\text{ times}}{\mathbb{Z}_p \times \dots \times \mathbb{Z}_p}$ is faithful and irreducible.
\end{proposition}

\begin{proof}
Note that $\begin{bmatrix}
1 & 0 & 0\\
a_1 & 1 & 0\\
b_1 & c_1 & 1 \end{bmatrix} \cdot \begin{bmatrix}
1 & 0 & 0\\
a_2 & 1 & 0\\
b_2 & c_2 & 1 \end{bmatrix} = \begin{bmatrix}
1 & 0 & 0\\
a_1 + a_2 & 1 & 0\\
b_1 + b_2 + c_1a_2 & c_1 + c_2 & 1 \end{bmatrix}$, and note that $Z(P) = \Biggl\{ \begin{bmatrix}
1 & 0 & 0\\
0 & 1 & 0\\
b & 0 & 1 \end{bmatrix} \,\, | \,\, b \in GF(p^n) \Biggr\} \cong \underset{n\text{ times}}{\mathbb{Z}_p \times \dots \times \mathbb{Z}_p}$.  Let $x$ be a generator of the group of units of $GF(p^n)$.  Define $\Bigg( \begin{bmatrix}
1 & 0 & 0\\
a & 1 & 0\\
b & c & 1 \end{bmatrix} \Bigg)^r = \begin{bmatrix}
1 & 0 & 0\\
x \cdot a & 1 & 0\\
x \cdot b & c & 1 \end{bmatrix}$.  Note that $r \in Aut(P)$ and $H = \langle r \rangle$ has order $p^n - 1$.  Define $\Bigg( \begin{bmatrix}
1 & 0 & 0\\
a & 1 & 0\\
b & c & 1 \end{bmatrix} \Bigg)^s = \begin{bmatrix}
1 & 0 & 0\\
a^p & 1 & 0\\
b^p & c^p & 1 \end{bmatrix}$.  Note that $s \in Aut(P)$ and $K = \langle s \rangle$ has order $n$.  Note that $s^{-1}rs = r^p$, and $T = [H]K \leq Aut(P)$ and acts faithfully and irreducibly on $Z(P)$.
\end{proof}

Proposition 9 with $p=2$ and $n=2$ yields examples of groups $P_1,P_2$ and $T_1,T_2$ in Theorem 2.  Proposition 9 with $p=3$ and $n=2$ yields examples of groups $P_3$ and $T_3$ in Theorem 2.  Proposition 8 tells us that each $P_i \in \CD(P_i)$ and tells us the structure of $\CD(P_i)$.  

The following example was provided by Ben Brewster, and yields examples of groups $P_1,P_2$ and $T_1$, $T_2$ in Theorem 2.  Let $D = \{ (x,x,x) \,\, | \,\, x \in Z(Q_8) \}$, and let $P = (Q_8 \times Q_8 \times Q_8)/D$.  Consider the natural action of $T = S_3$ on $P$ via permutation of the coordinates.  Then $T$ acts faithfully and irreducibly on $Z(P) \cong \mathbb{Z}_2 \times \mathbb{Z}_2$.  It can be shown that $\CD(P) = \{ (X \times Y \times Z)/D \,\, | \,\, X,Y,Z \in \CD(Q_8) \}$.  The author verified this through use of GAP.  We know that $\CD(Q_8)$ consists of all of the subgroups of $Q_8$ that contain $Z(Q_8)$, and so $\CD(Q_8)$ is a quasi-antichain of width 3.  And so $\CD(P) \cong \mathcal{M}_3 \times \mathcal{M}_3 \times \mathcal{M}_3$.  The author would like to remark that $Q_8$ is an example of an extraspecial $p$-group, and it is true that for any extraspecial $p$-group, $R$, $\CD(R)$ consists of all of the the subgroups of $R$ that contain $Z(R)$.  This was mentioned in Example 2.8 in \cite{Gla06}; the result follows from considering a non-degenerate bilinear form induced by commutation that is endowed upon $R/Z(R)$ when viewed as a vector space over $\mathbb{F}_p$.  The author would like to remark that in a recent paper, \cite{Ta17}, the groups $G$ having the propery that $\CD(G)$ consists of all of the subgroups of $G$ that contain $Z(G)$ are classified.

We collect all of the above examples that work in Theorem 2, and we arrive at 12 examples of CD-minimal groups that are not CD-simple, with each example having different CD lattice.  The table below shows possibilities for $\CD(P_1)$, $\CD(P_2)$, and $\CD(P_3)$ in the construction of $S$.  And so $\CD(S) = (\CD(P_1) \times \CD(P_2) \times \CD(P_3)) \cup \{1, S\}$.  

\begin{table}[h]

  \begin{tabular}{|c|c|c|}
      \hline

      $\CD(P_1)$    &   $\CD(P_2)$     &   $\CD(P_3)$         \\ \hline

 $\mathcal{M}_0$    &   $\mathcal{M}_0$     &   $\mathcal{M}_0$         \\ \hline

$\mathcal{M}_0$    &   $\mathcal{M}_5$     &   $\mathcal{M}_0$         \\ \hline

$\mathcal{M}_0$    &   $\mathcal{M}_3 \times \mathcal{M}_3 \times \mathcal{M}_3$     &   $\mathcal{M}_0$         \\ \hline

 $\mathcal{M}_5$    &   $\mathcal{M}_5$     &   $\mathcal{M}_0$         \\ \hline

 $\mathcal{M}_5$    &   $\mathcal{M}_3 \times \mathcal{M}_3 \times \mathcal{M}_3$     &   $\mathcal{M}_0$         \\ \hline

 $\mathcal{M}_3 \times \mathcal{M}_3 \times \mathcal{M}_3$    &   $\mathcal{M}_3 \times \mathcal{M}_3 \times \mathcal{M}_3$     &   $\mathcal{M}_0$         \\ \hline

  $\mathcal{M}_0$    &   $\mathcal{M}_0$     &   $\mathcal{M}_{10}$         \\ \hline

$\mathcal{M}_0$    &   $\mathcal{M}_5$     &   $\mathcal{M}_{10}$         \\ \hline

$\mathcal{M}_0$    &   $\mathcal{M}_3 \times \mathcal{M}_3 \times \mathcal{M}_3$     &   $\mathcal{M}_{10}$         \\ \hline

 $\mathcal{M}_5$    &   $\mathcal{M}_5$     &   $\mathcal{M}_{10}$         \\ \hline

 $\mathcal{M}_5$    &   $\mathcal{M}_3 \times \mathcal{M}_3 \times \mathcal{M}_3$     &   $\mathcal{M}_{10}$         \\ \hline

 $\mathcal{M}_3 \times \mathcal{M}_3 \times \mathcal{M}_3$    &   $\mathcal{M}_3 \times \mathcal{M}_3 \times \mathcal{M}_3$     &   $\mathcal{M}_{10}$         \\ \hline

\end{tabular}
\end{table}

We end with some open questions.  Obviously, one desires more theory on the class of groups $\mathcal{T}$ and the class of lattices $\mathcal{T}^*$, and more examples of groups and lattices in each class would further that theory.  Is the appearance of quasi-antichain $p$-groups lattices in this construction mere coincidence, perhaps because of the small order of groups involved in the construction, or is there a deeper reason that quasi-antichain lattices appear prevalently in CD lattices?   For groups in $\mathcal{T}$ and lattices in $\mathcal{T}^*$, we have established a relationship between direct products of groups and Cartesian products of lattices.  Are there more relationships between lattice theoretic properties of $\CD(G)$ and group theoretic properties of $G$? 

\bibliographystyle{amsplain}
\bibliography{references}

\end{document}